\documentclass[12pt, a4paper]{amsart}
\usepackage[utf8]{inputenc}
\usepackage{latexsym}
\usepackage{amscd, amsfonts,
  eucal, mathrsfs, amsmath, amssymb, amsthm}
\input xy
\xyoption{all}

\usepackage{fullpage}
\usepackage{tikz-cd}
\usepackage{extarrows}
\tikzset{
  labl/.style={anchor=south, rotate=90, inner sep=.5mm}
}
\tikzset{
  labr/.style={anchor=north, rotate=90, inner sep=1mm}
}

\usepackage{mathtools}
\usepackage{bm}

\usepackage{array}
\usepackage{color}

\usepackage[cal=boondoxo, scr=boondoxo]{mathalfa}

\usepackage[sort,nocompress]{cite}

\usepackage{hyperref}

\DeclareMathOperator{\End}{End}

\newcommand{\field}[1]{\mathbb #1}

\renewcommand{\phi}{\varphi}

\newcommand{\C}{\field C}

\newcommand{\Z}{\field Z}
\newcommand{\Q}{\field Q}
\newcommand{\N}{\field N}

\renewcommand{\P}{\field P}

\DeclareMathOperator{\Aut}{\operatorname{Aut}}

\newtheoremstyle{fauxsubsub}
  {3pt}
  {0pt}
  {}
  {0pt}
  {}
  {.}%
  {.5em}
  {}

\newtheorem{lem}{Lemma}[section]

\numberwithin{equation}{lem}

\newtheorem{thm}[lem]{Theorem}
\newtheorem*{mainthm}{Main Theorem}
\newtheorem*{corollary*}{Corollary}

\newtheorem{prop}[lem]{Proposition}

\newtheorem{cor}[lem]{Corollary}

\theoremstyle{definition}
\newtheorem{defn}[lem]{Definition}

\newtheorem{example}[lem]{Example}

\theoremstyle{definition}

\newtheorem*{remark*}{Remark}
\newtheorem*{remarks*}{Remarks}

\newtheorem{ques}[lem]{Question}

\theoremstyle{fauxsubsub}

\numberwithin{equation}{lem}

\author{Yuya Sasaki}
\title{Nonnatural automorphisms of the Hilbert scheme of two points of
some simple abelian varieties}

\begin{document}

\begin{abstract}
  We construct nonnatural automorphisms of the Hilbert scheme of two
  points of some simple abelian varieties preserving the big diagonal
  by considering automorphisms of the $n$-th product of the abelian
  varieties. 
\end{abstract}

\maketitle

\section{Introduction}

Throughout this paper, we work over the field $\C$ of complex numbers.
Let $X$ be a smooth projective variety, and for $n \in \N_{\ge 2}$,
let $X^{[ n ]}$ be the Hilbert scheme of $n$ points on $X$. 
An automorphism $\sigma$ of $X^{[ n ]}$ is called natural if
$\sigma = f^{[ n ]}$ for some $f \in \Aut ( X )$. This terminology seemes
to originate in a paper of Boissi\`{e}re \cite{Boi12}, and there are some
interesting problems with this naturality, for example, the
existence of nonnatural automorphisms and an explicit description of them.
When $X$ is a K3 surface, there may exist nonnatural automorphisms,
see \cite{BCNWS16, Cat19, Ogu16, Zuf19}.
When $X$ is a weak Fano surface or a surface of general type,
\cite{BOR20} proves that there does not exist such an automorphism
except in the case where $n = 2$ and $X$ is a product of smooth curves
(cf. \cite{Hay20} also studied the case of weakly del Pezzo surfaces as
well as rational elliptic surfaces).
In \cite{BOR20}, it is
also proved that there exists no nonnatural automorphism when $X =
\P^m$ for $m \ge 2$ and $n = 2$, and this result is generalized to
the case where $X \subset \P^m$ is a hypersurface of degree $d \ge 2$ for
$n = 2$, $m \ge 4$ in \cite{Wang23}. Besides these, there are some
problems concerning naturality (cf. \cite{HT16}).

One of such problems is the relation between naturality and the big
diagonal. First, we define the big diagonal $\Delta \subset X^n$
as follows: \[
  \Delta := \left\{ \left( x_1, \ldots, x_n \right) | \exists i, j
  \ \text{with} \ i \neq j, \ x_i = x_j \right\}
\]
and call this the big diagonal of $X^n$. Also, we define the big diagonal
$\Delta^{( n )}$ of $X^{( n )}$ to be the image of $\Delta$
by the quotient map $X^n \to X^{( n )}$ and the big
diagonal $\Delta^{[ n ]}$ of $X^{[ n ]}$ by $\Delta^{[ n ]} :=
\epsilon^{- 1} \left( \Delta^{( n )} \right)$ with $\epsilon$ the
Hilbert-Chow morphism. Then, $\Delta^{[ n ]}$ is the exceptional
divisor of $\epsilon$. 
The following question is posed in \cite{BOR20}:

\begin{ques}[{\cite[Question 3]{BOR20}}]
  Suppose $X$ is a smooth projective surface and $\sigma \colon
  X^{[ n ]} \to X^{[ n ]}$ is an automorphism preserving the big
  diagonal $\Delta^{[ n ]}$. Except for the case $\left( C_1 \times C_2
  \right)^{[ 2 ]}$, where $C_1$ and $C_2$ are smooth projective curves
  and $n = 2$, does it follow that $\sigma$ is natural?
  \label{bor}
\end{ques}

It is known that this holds for K3 surfaces as a result of
Boissi\`{e}re and Sarti \cite{BS12} and for Enriques surfaces by Hayashi
\cite{Hay17}.
Also, it is proved by \cite{BOR20} that this holds for weak del Pezzo
surfaces and surfaces of general type. 
In the case where $n = 2$ and $X$ is a product of smooth curves, this
does not hold, especially when either both curves are rational or both
of genus $g \ge 2$, which is also proved in \cite{BOR20}.
Also, there is a similar result for a generalized Kummer variety, see 
\cite{BNWS10}.

When we consider this question, the following proposition is useful:

\begin{prop}[{\cite[Proposition 9, Proposition 12]{BOR20}}]
  Let $X$ be a smooth projective surface, and $n \ge 2$. Then, for
  every $\tau \in \Aut \left( X^{( n )} \right)$, there exists
  $F \in \Aut \left( X^n \right)$ such that $\tau \circ \rho = \rho
  \circ F$ with $\rho$ the quotient map $X^n \to X^{( n )}$. 
  \label{borp}
\end{prop}

By this proposition, for $\sigma \in \Aut \left( X^{[ n ]} \right)$,
if $\sigma$ descends to
an automorphism $\tau$ of $X^{( n )}$, $\sigma$ is natural if and only
if the automorphism $F$ of $X^n$ induced by $\sigma$ is of the form \[
  F = \alpha \circ \left( f \times \cdots \times f \right)
\]
with $\alpha \in \mathfrak{S}_n$ and $f \in \Aut X$ by the definition of
the symmetric product. Therefore,
considering automorphisms of $X^n$ satisfying the following property
can be important in Question \ref{bor}, or the general naturality problem: 

\begin{defn}
  Let $F \in \Aut \left( X^n \right)$. Then,
  $F$ is said to satisfy \textit{the symmetric condition} if it descends
  to an automorphism of $X^{( n )}$, i.e., there exists $\tau \in
  \Aut \left( X^{( n )}\right)$ such that $\tau \circ \rho = \rho
  \circ F$.
\end{defn}

It is clear that an automorphism $F$ of $X^n$ in Proposition \ref{borp}
satisfies the symmetric condition and that $F \in \Aut \left( X^n \right)$
satisfies $F ( \Delta ) = \Delta$, i.e., $F$ preserves the big diagonal.
Also, in \cite{BOR20}, these two properties are used to prove the
existence of a nonnatural automorphism when $X$ is a product of curves.

In this paper, we consider the question where $X$ is a simple abelian
variety, and first, we consider automorphisms of $X^n$ satisfying
the symmetric condition. In this paper, when considering
automorphisms of abelian varieties, they are assumed to be group
homomorphisms. 

\begin{thm}
  Let $X$ be a simple abelian variety and $n \in \N_{\ge 2}$. Then, all
  automorphisms of $X^n$ satisfying the symmetric condition are
  of the form \[
    F = \sigma \circ
    \begin{pmatrix}
      f & g & \cdots & g \\
      g & f & \ddots& \vdots \\
      \vdots & \ddots & \ddots & g \\
      g & \cdots & g & f
    \end{pmatrix}
  \]
  with $\sigma \in \mathfrak{S}_n$ a permutation matrix and $f, g \in
  \End ( X )$. 

  Conversely, all automorphisms of $X^n$ of the above form satisfy the
  symmetric condition. 
  \label{sub}
\end{thm}

Based on this, we construct the following example: 

\begin{mainthm}
  For $d \in \N_{\ge 2}$ and $n \in \N_{\ge 2}$, there exists a simple
  abelian variety $X$ of dimension $d$ and an automorphism $F$ of $X^n$
  such that $F$ satisfies the symmetric condition and is of the form \[
    F = \sigma \circ
    \begin{pmatrix}
      f & g & \cdots & g \\
      g & f & \ddots& \vdots \\
      \vdots & \ddots & \ddots & g \\
      g & \cdots & g & f
    \end{pmatrix}
  \]
  with $\sigma \in \mathfrak{S}_n$ a permutation matrix and $f, g \in
  \Aut ( X )$, $g \neq 0$.
  \label{main}
\end{mainthm}
 
Finally, as a corollary, we get a negative answer to Question \ref{bor}: 

\begin{cor}
  For $d \ge 2$, there exists a simple abelian variety $X$ of dimension
  $d$ and an automorphism $\sigma$ of $X^{[ 2 ]}$ such that
  $\sigma$ preserves the big diagonal and is nonnatural.

  When $d = 2$, this gives a negative answer to Question \ref{bor}.
  \label{cor}
\end{cor}

After some preliminaries in Section \ref{prel}, we prove
Theorem \ref{sub} in Section \ref{subs}, and the main theorem in
Section \ref{mains}, respectively.

\section{Preliminaries}
\label{prel}

\subsection{Hilbert schemes of $n$ points}

Let $X$ be a smooth projective variety. Let $X^{[ n ]}$ be the Hilbert
scheme of $n$ points of $X$, i.e., the Hilbert scheme of zero dimensional
subschemes of length $n$ of $X$, and $X^{( n )}$ be the $n$-th symmetric
product of $X$, i.e., the quotient of $X^n$ under the permutation action
by $\mathfrak{S}_n$. Then, we have a map $\epsilon \colon X^{[ n ]} \to
X^{( n )}$ sending a subscheme $Z \subset X$ to its support, which is
called the Hilbert-Chow morphism. When $n = 2$, $X^{( 2 )}$ is singular
precisely at the big diagonal $\Delta^{( 2 )}$ and the Hilbert-Chow
morphism is a blowup with the center $\Delta^{( 2 )}$, the big diagonal
of $X^{( 2 )}$. Also, when $X$ is a surface, $X^{[ n ]}$ is nonsingular
for any $n$ and $\Delta^{( n )}$ is the singular locus on $X^{( n )}$. 

\subsection{Endomorphisms of an abelian variety}

For an abelian variety $X$, we denote by $\End ( X )$ the endomorphism
ring of $X$, and by $\Aut ( X )$ the group of automorphisms of $X$
preserving $0$. Also, let $\End_\Q ( X ) := \End ( X ) \otimes_\Z \Q$. 
An abelian variety is called simple if it has no nontrivial abelian
subvariety. Then, the following holds: 

\begin{thm}[{\cite[Theorem 5.3.7, Corollary 5.3.8]{BL04}}]
  Given an abelian variety $X$, there is an isogeny \[
    X \to X_1^{\times n_1} \times \cdots \times X_r^{\times n_r}
  \]
  with simple abelian varieties $X_\nu$ not isogenous to each other. 
  Moreover, the abelian varieties $X_\nu$ and the integers $n_\nu$ are
  uniquely determined up to isogenies and permutations.

  Also, in this condition, \[
    \End_\Q ( X ) \cong M_{n_1} ( F_1 ) \oplus \cdots \oplus
    M_{n_r} ( F_r )
  \]
  where $F_\nu = \End_\Q ( X_\nu )$ and $M_{n} ( R )$ is the matrix ring
  of size $n \times n$ with coefficients in a ring $R$. 
  \label{pcr}
\end{thm}

Therefore, in order to simplify the study of endomorphisms,
we will consider simple abelian varieties. 

\section{Automorphisms of the $n$-th product of a simple abelian variety
satisfying the symmetric condition}
\label{subs}

Let $X$ be a simple abelian variety. In this section, we consider
automorphisms of $X^n$ satisfying the symmetric condition.
Then, it is enough to consider automorphisms of the type \[
  F = \left( f_{i j} \right) : X^n \to X^n
\]
where $f_{i j} : X_j \to X_i$ is an endomorphism and $X_i = X$ is the
$i$-th factor of $X^n$. In fact, an endomorphism $F \colon X^n \to X^n$
is determined by $F_i \colon X^n \to X_i$ for each $i$, and as \[
  F_i ( x_1, \ldots, x_n ) = F_i ( x_1, 0, \ldots, 0) + F_i ( 0, x_2, 0,
  \ldots, 0 ) + \cdots F_i ( 0, \ldots, 0, x_n ),
\]
any endomorphisms $F$ of $X^n$ is written in the above form
(cf. \cite[Chapter 1.2]{BL04}). Note that $F$ can be an automorphism even
if not all $f_{i j}$'s are automorphisms. 

\begin{proof}[Proof of Theorem \ref{sub}]
  For $\sigma, \tau \in \mathfrak{S}_n$, \[
    \sigma \circ
    \begin{pmatrix}
      f & g & \cdots & g \\
      g & f & \ddots & \vdots \\
      \vdots & \ddots & \ddots & g \\
      g & \cdots & g & f
    \end{pmatrix}
    \circ \tau = \sigma \circ \tau \circ
    \begin{pmatrix}
      f & g & \cdots & g \\
      g & f & \ddots & \vdots \\
      \vdots & \ddots & \ddots & g \\
      g & \cdots & g & f
    \end{pmatrix},
  \]
  Therefore, the latter statement holds. So we will prove the former. 
  Let $F =
  \left( f_{i j} \right)$ be an automorphism of $X^n$ satisfying the
  symmetric condition. Also, for $n \ge 3$, we define $\Delta' \subset
  X^n$ as follows: \[
    \Delta' := \left\{ \left( x_1, \ldots, x_n \right) \in X^n |
    \exists i \ \text{with}\ 1 \le i \le n;\ x_1 = \cdots =
    x_{i - 1} = x_{i + 1} = \cdots = x_n \right\}.
  \]
  Note that $\Delta = \Delta'$ when $n = 3$, and as $F$ is an automorphism
  preserving the big diagonal, $F$ also preserves $\Delta'$. 

  First, by the symmetric condition, the images of \[
    F \left( 0, \ldots, 0, x, 0, \ldots, 0 \right) = \left( f_{1 i}
    ( x ), \ldots, f_{n i} ( x ) \right)
  \]
  by the quotient map $X^n \to X^{( n )}$ agree for all $1 \le i \le n$
  as the images of $( 0, \ldots, 0, x, 0, \ldots, 0 )$ by $X^n \to
  X^{( n )}$ agree.

  Thus, for each $x \in X$ and $1 \le i, j \le n$, there exists $1 \le
  k' \le n$ such that \[
    f_{i 1} ( x ) = f_{k' j} ( x )
  \]
  and this implies \[
    \bigcup_{k = 1}^n \left( f_{i 1} = f_{k j} \right) = X.
  \]
  As each $\left( f_{1 i} = f_{k j} \right) \subset X$ is closed, there
  exists $1 \le k \le n$ such that $f_{1 i} = f_{k j}$, and applying the
  same argument to $f_{j i}$ for each $j$, we see that the sets \[
    \left\{ f \in \End ( X ) | \exists j \ \text{with}\ 1 \le j \le n,\ f
    = f_{j i} \right\}
  \]
  agree for all $1 \le i \le n$.  Also, by a similar argument,
  we see that $f_{1 i}, \ldots, f_{n i}$ is a permutation of $f_{1 1},
  \ldots, f_{n 1}$ for any $2 \le i \le n$. Note that there exist $1 \le
  i, j \le n,\ i \neq j$ such that $f_{1 i} \neq f_{1 j}$ as $F$ is an
  automorphism, and thus, the theorem for $n = 2$ follows from these. 

  Now, suppose $n \ge 3$. Then, as $F$ preserves $\Delta'$, \[
    F \left( x, 0, \ldots, 0 \right) = \left( f_{1 1} ( x ),
    \ldots, f_{n 1} ( x ) \right)
  \]
  is in $\Delta'$. Therefore, for each $x \in X$, there exists $1 \le k'
  \le n$ such that \[
    f_{1 1} ( x ) = \cdots = f_{k' - 1 \ 1} ( x ) = f_{k' + 1 \ 1} ( x ) =
    \cdots f_{n 1} ( x ).
  \]
  Therefore, we see that \[
    \bigcup_{1 \le k \le n} \left( f_{1 1} = \cdots = f_{k - 1 \ 1} =
    f_{k + 1 \ 1} = \cdots = f_{n 1} \right) = X,
  \]
  and by the similar argument as above, 
  we see that there exists $1 \le k \le n$
  such that for any $i', j' \neq k$, \[
    f_{i' 1} = f_{j' 1}.
  \]
  Now, suppose $f_{1 i}, \ldots, f_{n i}$ is a permutation of $f, g,
  \ldots, g$ with $f, g \in \End ( X )$ for all $1 \le i \le n$. Then, 
  as $F$ is an automorphism, $f \neq g$, and for $i \neq j$, $f_{k i}
  = f$ implies $f_{k j} = g$, otherwise it cannot be an automorphism. 
  Therefore, $F$ is of the form \[
    \sigma \circ
    \begin{pmatrix}
      f & g & \cdots & g \\
      g & f & \ddots & \vdots \\
      \vdots & \ddots & \ddots & g \\
      g & \cdots & g & f
    \end{pmatrix}
  \]
  with $\sigma \in \mathfrak{S}_n$.
\end{proof}


\section{Proof of the main theorem}
\label{mains}

\begin{proof}[Proof of the main theorem]
  Take a totally real number field $K \supset \Q$ of degree $d$. Then,
  we can take a simple abelian variety $X$ of dimension $d$ such that
  $\End_\Q ( X ) \cong K$ (cf. \cite[Chapter 9.2]{BL04}).
  We identify endomorphisms of $X$ with elements of $K$ by this
  isomorphism. Then, an endomorphism of $X$ is integral over $\Z$ as
  an element of $K$, so $\End ( X ) \subset O_K$ where $O_K$ is the
  ring of integers of $K$.
  As $K$ is totally real, there exists $\alpha \in O_K^\times$
  such that $\alpha > 1$ by Dirichlet's unit theorem.
  Then, as $\End ( X ) \subset O_K$ is of finite index (cf.
  \cite[Chapter 1.2]{BL04}, \cite[Chapter 1.2]{Neu99}), we may assume
  $\alpha \in \End ( X )$ by replacing $\alpha$ by a suitable power. 

  Then, it is enough to find $f, g \in \End
  ( X )$ such that $g \neq 0$ and \[
    \begin{pmatrix}
      f & g & \cdots & g \\
      g & f & \ddots& \vdots \\
      \vdots & \ddots & \ddots & g \\
      g & \cdots & g & f
    \end{pmatrix}
  \]
  is an automorphism of $X^n$(cf. \ref{sub}),
  and as the determinant of this matrix is \[
    \left( f - g \right)^{n - 1} \left( f + \left( n - 1 \right) g
    \right), 
  \]
  it is enough to find $f, g \in \End ( X )$ such that $g \neq 0$,
  $f - g, f + \left( n - 1 \right) g \in \Aut ( X ) = \End ( X)
  \cap O_K^\times$. 

  As $\End ( X )$ is a free $\Z$-module, let $a_0, a_1, \ldots,
  a_{d - 1}$ be a $\Z$-basis of $\End ( X )$. 
  Then, there exists $i \neq j \in \Z$ such that
  \begin{align*}
    & \alpha^i = b_0 \cdot a_0 + \cdots + b_{d - 1} a_{d - 1} \\
    & \alpha^j = c_0 \cdot a_0 + \cdots + c_{d - 1} a_{d - 1} \\
    & b_i, c_i \in \Z,\ b_i - c_i \equiv 0 \bmod n \ \ \left( 0 \le \forall
    i \le n \right)
  \end{align*}
  as combinations of coefficients of $a_i$'s modulo $n$ are finite. Then,
  there exist $f, g \in \End ( X )$ such that \[
    \begin{cases}
      f + \left( n - 1 \right) g & = \alpha^i \\
      f - g & = \alpha^j.
    \end{cases}
  \]
  In fact, \[
    \begin{cases}
      f & = \dfrac{\alpha^i + \left( n - 1 \right) \alpha^j}{n}
      = \alpha_j - \sum_i \dfrac{b_i - c_i}{n} a_i \\
      g & = \dfrac{\alpha^i - \alpha^j}{n} = \sum_i \dfrac{b_i - c_i}{n}
      a_i
    \end{cases}
  \]
  are both in $\End ( X )$, and the statement holds from this. 

\end{proof}

By applying this theorem to the $n = 2$ case, the above automorphism of
$X^2$ induces an automorphism of $X^{[ 2 ]}$ which is nonnatural
as the Hilbert-Chow morphism is a blow-up with center $\Delta^{( 2 )}$,
which proves Corollary \ref{cor}. 

Finally, we give an explicit example when $d = 2$ and $n = 2$. 

\begin{example}
  Let $K := \Q ( \sqrt{2} )$ and X be a simple abelian surface
  such that $\End_\Q ( X ) \cong K$. As in the proof of the main theorem,
  we identify an endomorphism of $X$ with an element of $K$ by this
  isomorphism. As $1 + \sqrt{2} \in O_K^\times$,
  there exists $i \in \Z_{> 0}$ such that $\left(
  1 + \sqrt{2} \right)^i \in \End \left( X \right)$. Then, as \[
    \left( 1 + \sqrt{2} \right)^i + \left( 1 - \sqrt{2} \right)^i \in
    2 \Z,
  \]
  we get a desired automorphism by setting $f, g \in \End ( X )$ so that
  \[
    \begin{cases}
      f & = \frac{1}{2} \left( \left( 1 + \sqrt{2} \right)^i +
      \left( 1 - \sqrt{2} \right)^i\right) \\
      g & = \frac{1}{2} \left( \left( 1 + \sqrt{2} \right)^i -
      \left( 1 - \sqrt{2} \right)^i \right).
    \end{cases}
  \]
\end{example}

\bibliographystyle{alpha}

\begin{thebibliography}{10}

  \bibitem[Boi12]{Boi12}
    S.~Boissi\`{e}re. Automorphismes naturels de l'espace de {D}ouady de
    points sur une surface. Canad. J. Math. {\bf 64} (2012), no. 1,
    3--23.

  \bibitem[BCNWS16]{BCNWS16}
    S.~Boissi\`{e}re, A.~Cattaneo, M.~Nieper-Wisskirchen, A.~Sarti.
    \newblock The automorphism group of the {H}ilbert scheme of two points
    on a generic projective {K}3 surface. In: {\em K3 Surface and their
    Moduli}, Progress in Mathematics, {\bf 315} 1--15.
    \newblock Cham: Birkh\"{a}user, 2016. 

  \bibitem[BL04]{BL04}
    C.~Birkenhake, H.~Lange.
    \newblock \textit{Complex abelian varieties}, second ed., Grundlehren
    der mathematischen Wissenschaften, {\bf 302}
    \newblock Springer-Verlag, Berlin, 2004.
  
  \bibitem[BNWS10]{BNWS10}
    S.~Boissi\`{e}re, M.~Nieper-Wisskirchen, A.~Sarti.
    \newblock Higher dimensional Enriques varieties and automorphisms of
    generalized Kummer varieties. 
    \newblock

  \bibitem[BOR20]{BOR20}
    P.~Belmans, G.~Oberdieck, J.~Rennemo.
    \newblock Automorphisms of {H}ilbert schemes of points on surfaces.
    \newblock Trans. Amer. Math. Soc. {\bf 373} (2020), no. 9, 6139--6156.

  \bibitem[BS12]{BS12}
    S.~Boissi\`{e}re, A.~Sarti.
    \newblock {\em A note on automorphisms and birational transformations
    of holomorphic symplectic manifolds}.
    \newblock Proc. Amer. Math. Soc. {\bf 140} (2012), no. 12, 4053--4062.

  \bibitem[Cat19]{Cat19}
    A.~Cattaneo.
    \newblock Automorphisms of {H}ilbert schemes of points on a generic
    projective {K}3 surface.
    \newblock Math. Nachr. {\bf 292} (2019), no. 10, 2137--2152. 

  \bibitem[HT16]{HT16}
    B.~Hassett, Y.~Tschinkel.
    \newblock Extremal rays and automorphimss of holomorphic symplectic
    varieties. In: {\em K3 surfaces and their moduli}, Progress in
    Mathematics, {\bf 315}, 73--95.
    \newblock Cham: Birkh\"{a}user, 2016. 

  \bibitem[Hay17]{Hay17}
    T.~Hayashi.
    \newblock {\em Universal covering Calabi-Yau manifolds of the Hilbert
    schemes of n-points of Enriques surfaces}.
    \newblock Asian J. Math. {\bf 21} (2017), no. 6, 1099--1120.

  \bibitem[Hay20]{Hay20}
    T.~Hayashi.
    \newblock Automorphisms of the {H}ilbert schemes of {$n$} points of
    a rational surface and the anticanonical {I}itaka dimension. 
    \newblock Geom. Dedicata, {\bf 207} (2020), 395--407. 

  \bibitem[Neu99]{Neu99}
    J.~Neukirch.
    \newblock \textit{Algebraic number theory}, Grundlehren der
    mathmatischen Wissenschaften, {\bf 322}.
    \newblock Springer-Verlag, Berlin, 1999.

  \bibitem[Ogu16]{Ogu16}
    K.~Oguiso.
    \newblock On automorphisms of the punctural {H}ilbert scheme of {K}3
    surfaces. 
    \newblock Eur. J. Math, {\bf 2} (2016), no. 1, 246--261. 

  \bibitem[Wang23]{Wang23}
    L.~Wang.
    \newblock {On automorphisms of {H}ilbert squares of smooth
    hypersurfaces}.
    \newblock Comm. Algebra, {\bf 51} (2023), no. 2, 586--595.

  \bibitem[Zuf19]{Zuf19}
    R.~Zuffetti.
    \newblock Strongly ambiguous {H}ilbert squares of projective {K}3
    surfaces with {P}icard number one.
    \newblock Rend. Semin. Mat. Univ. Politec. Torino, {\bf 77} (2019),
    no. 1, 113--130. 


\end{thebibliography}

\end{document}